\documentclass[12pt]{amsart}
\usepackage[numbers, square]{natbib}
\usepackage{amssymb}
\usepackage{amsmath}
\usepackage{amsfonts}
\usepackage{geometry}
\usepackage{graphicx}
\usepackage{amsthm}
\usepackage{hyperref}
\usepackage{esint}
\usepackage[dvipsnames]{xcolor}

\providecommand{\U}[1]{\protect\rule{.1in}{.1in}}
\theoremstyle{plain}

\newtheorem{corollary}{Corollary}

\newtheorem{remark}{Remark}

\newtheorem{theorem}{Theorem}
\numberwithin{equation}{section}

\begin{document}
\title[Propagation of smallness in elliptic periodic homogenization ]
{ propagation of smallness in elliptic periodic homogenization  }
\author {Carlos Kenig}
\address{
Department of Mathematics\\
University of Chicago\\
Chicago, IL 60637, USA\\
Email:  cek@math.uchicago.edu }
\author{ Jiuyi Zhu}
\address{
Department of Mathematics\\
Louisiana State University\\
Baton Rouge, LA 70803, USA\\
Email:  zhu@math.lsu.edu }

\thanks{Kenig is supported in part by NSF grant DMS-1800082,
Zhu is supported in part by  NSF grant OIA-1832961}
\date{}
\subjclass[2010]{35B27, 35J61, 35C15.} \keywords {Homogenization, Three-ball theorem, Approximate propagation of smallness, Poisson kernel}

\begin{abstract}
The paper is mainly concerned with an approximate three-ball inequality for solutions in elliptic periodic homogenization. We consider a family of second order operators $\mathcal{L}_\epsilon$ in divergence form with rapidly oscillating and periodic coefficients. It is the first time  such an approximate three-ball inequality  for homogenization theory is obtained. It implies an approximate quantitative propagation of smallness. The proof relies on a representation of the solution by the Poisson kernel and the Lagrange  interpolation technique. Another full propagation of smallness result is also shown.
\end{abstract}

\maketitle
\section{Introduction}
Quantitative propagation of smallness plays an important role in the quantitative study of solutions of elliptic and parabolic equations. It can be stated as follows:  a solution $u$ of a PDE $Lu=0$ on a domain $X$ can be made arbitrarily small on any given compact subset of $X$ by making it sufficiently small on an arbitrary given subdomain $Y$. The quantitative propagation of smallness has found many important applications, such as the stability estimates for the Cauchy problem \cite{ARRV} and the Hausdorff measure estimates of nodal sets of eigenfunctions \cite{Lin}, \cite{Lo}. Hadamard's three-circle or three-ball theorem  is the simplest quantitative statement for propagation of smallness. The Hadamard three-ball theorem ( or three-ball inequality) for a function $u$ is the inequality
\begin{align}
\|u\|_{r_2}\leq C \|u\|^\alpha_{r_1}\|u\|^{1-\alpha}_{R},
\label{thball}
\end{align}
where $0<r_1<r_2<R$ and $\|\cdot\|_{r}$ is the $L^2$ or $L^\infty$ norm on the ball centered at the origin with radius $r$, $\alpha\in (0,\ 1)$ and $C$ is a constant depending on $r_1$, $r_2$, and $R$. We will call the inequality (\ref{thball}) the standard three-ball inequality.
For holomorphic functions, the three-ball inequality (\ref{thball}) is a consequence of the convexity of the logarithm of the modulus of holomorphic functions and the maximum principle. The strategy of using logarithmic convexity of the modulus has played a central role in obtaining three-ball theorems (or three-ball inequalities).

For solutions of second order elliptic equations
 \begin{align}
Lu= -D_j(a_{ij}(x)D_i u)+ b_i D_i u +c(x) u=0,
\label{second}
\end{align}
(the summation convention is used throughout the paper), the logarithmic convexity of the $L^2$-norm of solutions has been applied in different ways to obtain the three-ball inequality by Agmon \cite{Ag} (for $L^2$-norm) and Landis \cite{L} (for $L^\infty$-norm).
For harmonic functions in higher dimensions, Korevaar and Meyers \cite{KM} obtained the three-ball inequality in (\ref{thball}) with sharp coefficient $C=1$  using the logarithmic convexity of the $L^2$ norm and the expansion on eigenfunctions of the Laplace-Beltrami operator on the sphere. Brummelhuis \cite{B} further pushed this technique to prove three-ball inequalities for solutions to second order elliptic equations.

The frequency function and Carleman estimates are  tools often used to obtain a three-ball inequality. The frequency function is given as a quotient involving $L^2$-norms of solutions. The monotonicity of the frequency function implies the convexity of the logarithm of $L^2$-norm of solutions, see e.g. \cite{Al}, \cite{GL}, \cite{Z}. Carleman estimates are weighted integral inequalities with suitable weight functions satisfying some convexity  properties. The three-ball inequality is obtained by applying the Carleman estimates to the product of the solution and an appropriate cut-off function and then by choosing the Carleman parameter appropriately, see e.g. \cite{DF}, \cite{JL}, \cite{K}. Recently, a proof of the three-ball inequality using the Poisson kernel for harmonic functions  was developed by Guadie and Malinnikova \cite{GM}.

The theory of homogenization identifies the average, macroscopic behavior of a phenomenon that is subject to microscopic effects. Homogenization not only has important impacts in applications, e.g. material sciences and elasticity theory,  but also shows potential applications to other areas of mathematics, such as  stochastic homogenization and numerical homogenization.
 In order to obtain the standard three-ball inequality (\ref{thball}), the Lipschitz continuity of the leading coefficient $a_{ij}(x)$ is required in the usual proof. This is consistent with the fact that there exist operators with H\"older continuous coefficients $a_{ij}$  with non-trivial solutions vanishing on open sets, which is precluded by (\ref{thball}), see e.g. \cite{M}.
 All the aforementioned techniques result in taking the derivative of the leading coefficient  $a_{ij}(x)$. The situation for the three-ball inequality for elliptic homogenization changes drastically since the derivative of the leading coefficient $a_{ij}(\frac{x}{\epsilon})$ leads to a ``bad" term $\frac{1}{\epsilon}$. In this paper, we intend to study the propagation of smallness for solutions in elliptic periodic homogenization. To achieve this, the main goal of this paper is to address a version of the three-ball inequality for elliptic periodic homogenization.

 We consider a family of elliptic operators in divergence form with rapidly oscillating periodic coefficients
\begin{align}
\mathcal{L}_\epsilon u=-{\rm div}\big(A(\frac{x}{\epsilon})\nabla u_\epsilon\big) =0 \quad \mbox{in} \ \Omega,
\label{main}
\end{align}
where $\epsilon>0$, $\Omega$ is a bounded domain in $\mathbb R^d$, and $A(y)=\big( a_{ij}(y)\big)$ is a symmetric $d\times d$ matrix-valued function in $\mathbb R^d$ for $d\geq 2$. Assume that $A(y)$ satisfies the following assumptions:

(i) Ellipticity: For some $0<\lambda<1$ and all $y, \xi\in \mathbb R^d,$ it holds that
\begin{align}
\lambda |\xi|^2\leq \langle A(y)\xi, \ \xi \rangle\leq \frac{1}{\lambda} |\xi|^2. \quad \quad
\label{ellip}
\end{align}

(ii) Periodicity:
\begin{align}
A(y+z)=A(y) \quad \quad \mbox{for} \ y\in  \mathbb R^d \ \mbox{and} \ z\in \mathbb Z^d.
\label{perio}
\end{align}

(iii) H\"older continuity: There exist constants $\tau>0$ and $0<\mu<1$ such that
\begin{align}
|A(x)-A(y)|\leq \tau |x-y|^\mu
\label{holder}
\end{align}
for any $x, y\in \mathbb R^d$.

We are able to establish the following approximate three-ball inequality in ellipsoids. The definition of ellipsoids $E_{r}$ depending on the coefficients $A(y)$ is given in Section 2.
\begin{theorem}
Let $u_\epsilon$ be a solution of (\ref{main}) in $\mathbb B_{10}$. For $0<r_1<r_2<\frac{ R}{4}<1$,
\begin{align}
\sup_{E_{r_2}}|u_\epsilon|\leq C \big \{\frac{r_2}{R} (\sup_{E_{r_1}}|u_\epsilon|)^\alpha (\sup_{E_{R}}|u_\epsilon|)^{1-\alpha}+ \frac{R^2}{r^2_1}[\epsilon \ln (\epsilon^{-1}+2)]^\alpha \sup_{E_{R}}|u_\epsilon|\big\},
\end{align}
where $\alpha= \frac{\ln \frac{R}{2r_2}} {\ln \frac{R}{r_1} }$ and $C$ depends on $\lambda$ and $(\tau, \mu)$.
\end{theorem}

Let $\mathbb B_r(0)$ be the ball centered at the origin with radius $r$. We usually write it as $\mathbb B_r$ if the context is understood. Since balls are more convenient in applications than ellipsoids, a direct consequence of Theorem 1 is the following approximate three-ball inequality in balls.
\begin{corollary}
Let $u_\epsilon$ be a solution of (\ref{main}) in $\mathbb B_{10}$. For $0<R_1<R_2<\frac{\lambda R_3}{4}<\frac{\lambda}{4}$,
\begin{align}
\|u_\epsilon\|_{L^\infty(\mathbb B_{R_2})}\leq C\big\{\frac{ R_2}{R_3} \|u_\epsilon\|^\beta_{L^\infty(\mathbb B_{R_1})} \|u_\epsilon\|^{1-\beta}_{L^\infty(\mathbb B_{R_3})} +\frac{R^2_3}{R^2_1} [\epsilon \ln (\epsilon^{-1}+2)]^\beta \|u_\epsilon\|_{L^\infty(\mathbb B_{R_3})}\big\},
\label{result}
\end{align}
where $\beta=\frac{\ln \frac{\lambda R_3}{2R_2}} {\ln \frac{R_3}{R_1} }$ and $C$ depends on $\lambda$ and $(\tau, \mu)$.

\end{corollary}

Let us address some important issues about the difference between the standard three-ball inequality (\ref{thball}) and inequality (\ref{result}).
\begin{remark}
The three-ball inequality (\ref{result}) in Corollary 1 is different from the standard three-ball inequality in (\ref{thball}). The inequality (\ref{thball}) implies the weak unique continuation property, which states that the solution vanishes globally if it vanishes in an open set. However, the estimate (\ref{result}) does not.    As $\epsilon \to 0$, the inequality (\ref{result}) converges to the standard three-ball inequality in (\ref{thball}). Compared with the Lipschitz regularity needed to obtain the inequality (\ref{thball}), only H\"older continuity is imposed to obtain the inequality (\ref{result}).
\end{remark}

If we choose $R_1, R_2, R_3$ in an appropriate scale, e.g. let
 $R_1=r$, $R_2=2r$ and $R_3=\frac{9r}{\lambda}$, we have the following corollary.
\begin{corollary}
Let $u_\epsilon$ be a solution of (\ref{main}) in $\mathbb B_{10}$. For $r\leq \frac{\lambda}{9},$
\begin{align}
\|u_\epsilon\|_{L^\infty(\mathbb B_{2r})}\leq C\big\{ \|u_\epsilon\|^\beta_{L^\infty(\mathbb B_{r})} \|u_\epsilon\|^{1-\beta}_{L^\infty(\mathbb B_{\frac{9r}{\lambda}})} + [\epsilon \ln (\epsilon^{-1}+2)]^\beta \|u_\epsilon\|_{L^\infty(\mathbb B_{\frac{9r}{\lambda}})}\big\},
\label{result2}
\end{align}
where  $\beta=\frac{\ln \frac{9}{4}} {\ln \frac{9}{\lambda}}$ and $C$ depends on $\lambda$ and $(\tau, \mu)$.
\end{corollary}
One can consider the eigenvalue type equation
\begin{align}
-{\rm div}\big(A(\frac{x}{\epsilon})\nabla u_\epsilon \big) =\lambda_{\epsilon, k} u_\epsilon  \quad \quad \mbox{in} \ \mathbb B_{10}.
\label{eigen}
\end{align}
Assume that $\lambda_{\epsilon, k}$ are positive constants and $\lambda_{\epsilon, k}\to\infty$ as $k\to \infty$. The following corollary holds.
\begin{corollary}
Let $u_\epsilon$ be a solution of (\ref{eigen}) in $\mathbb B_{10}$. For $0<R_1<R_2<\frac{\lambda R_3}{4}<\frac{\lambda}{4}$,
\begin{align}
\|u_\epsilon\|_{L^\infty(\mathbb B_{R_2})}&\leq C e^{2R_3 \sqrt{\lambda_{\epsilon, k}} }\big\{\frac{ R_2}{R_3} \|u_\epsilon\|^\beta_{L^\infty(\mathbb B_{R_1})} \|u_\epsilon\|^{1-\beta}_{L^\infty(\mathbb B_{R_3})} \nonumber \\ &+\frac{R^2_3}{R^2_1} [\epsilon \ln (\epsilon^{-1}+2)]^\beta \|u_\epsilon\|_{L^\infty(\mathbb B_{R_3})}\big\},
\end{align}
where $\beta=\frac{\ln \frac{\lambda R_3}{2R_2}} {\ln \frac{R_3}{R_1} }$ and $C$ depends on $\lambda$ and $(\tau, \mu)$.
\label{cor2}
\end{corollary}

Thanks to the three-ball inequality (\ref{result2}), we can show an approximate quantitative propagation of smallness result for solutions of (\ref{main}).
\begin{corollary}
Let $u_\epsilon$ be a solution of (\ref{main}) in $\mathbb B_{10}$. Assume that $\|u_\epsilon\|_{L^\infty(\mathbb B_{r})}\leq \delta$ for some small $r<\frac{\lambda}{9}$ and $\|u_\epsilon\|_{L^\infty(\mathbb B_{10})}\leq 1$. Then, with $\beta=\frac{\ln \frac{9}{4}} {\ln \frac{9}{\lambda}}$,
\begin{align}
\|u_\epsilon\|_{L^\infty(\mathbb B_{9})}\leq C^{\frac{1}{1-\beta}}\delta^{\beta^m}+ m C^{\frac{1}{1-\beta}}   [\epsilon \ln (\epsilon^{-1}+2)]^{\beta^m}
\label{progaga}
\end{align}
\label{cor3}
for some positive integer $m$ depending only on $r$, for $\epsilon \ln (\epsilon^{-1}+2)\leq 1. $
\end{corollary}
From the inequality (\ref{progaga}), we learn that $u_\epsilon$ in any compact domain of $\mathbb B_{10}$ can be made small
if $\delta$ is sufficiently small and $\epsilon$ is sufficiently small, with an explicit dependence of $\delta$ and $\epsilon$.

Recently, an interesting doubling inequality was shown by Lin and Shen \cite{LS} under a stronger regularity of $A(y)$. Assume that $A(y)$ satisfies the following Lipschitz continuous condition.\\
(iv) Lipschitz continuity: There exists a constant $\tau>0$ such that
\begin{align}
|A(x)-A(y)|\leq \tau |x-y| \quad \quad
\label{lips}
\end{align}
for any $x, y\in \mathbb R^d$.

Let $\fint_{\mathbb B_{r}}$ be the average integral in $\mathbb B_{r}$.  Under the assumptions of (\ref{ellip}), (\ref{perio}) and (\ref{lips}) for $A(y)$, the doubling inequality results in \cite{LS} can be formulated as

\textbf{Theorem A}: Let $u_\epsilon$ be a solution of (\ref{main}) in $\mathbb B_{{4}/{\sqrt{\lambda}}}$. Assume that
\begin{align}
\fint_{\mathbb B_{{2}/{\sqrt{\lambda}}}}u^2_\epsilon\, dx\leq N \fint_{\mathbb B_{\sqrt{\lambda}}}u^2_\epsilon\, dx
\end{align}
for some $N>1$. Then
\begin{align}
\fint_{\mathbb B_{r}}u^2_\epsilon\, dx\leq C(N) \fint_{\mathbb B_{\frac{r}{2}}}u^2_\epsilon\, dx
\end{align}
for any $0<r<1$, where $C(N)$ depends on $N, \lambda, \tau, d$.

Note that $C(N)$ can be chosen increasing with respect to $N$. By standard arguments, we can replace the average $L^2$ norm by  $L^\infty$ norm  in Theorem A. Thanks to Theorem {A}, we are able to show the following full propagation of smallness result.
\begin{corollary}
Assume that $\|u_\epsilon\|_{L^\infty(\mathbb B_{2/\sqrt{\lambda}})}\leq 1$. Then, given any constants $\eta_0>0$, $0<r_0<\frac{\sqrt{\lambda}}{2}$, there exists a positive constant $\delta_0=\delta_0(\lambda, d, \tau, r_0, \eta_0)$ such that if $\|u_\epsilon\|_{L^\infty(\mathbb B_{r_0})}\leq \delta_0$, then
$\|u_\epsilon\|_{L^\infty(\mathbb B_{\sqrt{\lambda}})}\leq \eta_0.$
\label{cornew}
\end{corollary}

Note that the dependence of $\delta_0$ is not explicit in the Corollary.
\begin{remark}
A direct consequence of Theorem A is the following doubling inequality with translation (Theorem 3.3 in \cite{LS}):
\begin{align}
\fint_{\mathbb B_{r}(x_0)}u^2_\epsilon\, dx\leq C(N) \fint_{\mathbb B_{\frac{r}{2}}(x_0)}u^2_\epsilon\, dx
\label{vanish}
\end{align}
for $0<r<\frac{3}{4}$ and $|x_0|\leq \frac{\sqrt{\lambda}}{2}$. We can also replace the average $L^2$ norm by $L^\infty$ norm in (\ref{vanish}).
Assume that $\|u_\epsilon\|_{L^\infty(\mathbb B_{\sqrt{\lambda}})}\geq 1$ and $\|u_\epsilon\|_{L^\infty(\mathbb B_{2/\sqrt{\lambda}})}\leq M$, a  uniform in $\epsilon$ vanishing order estimate for $u_\epsilon$ can be shown  in $\mathbb B_{\sqrt{\lambda}/2}$ by iteration of  (\ref{vanish}) in $L^\infty$ norm. However, such a vanishing order estimate is implicit since it depends on $C(M)$. See e.g. \cite{K}, \cite{Z}, for some literature on the explicit vanishing order estimates for solutions in (\ref{second}).
\end{remark}

We can extend the three-ball inequality (\ref{result}) to elliptic homogenization with Dirichlet or Neumann boundary conditions. We consider elliptic homogenization with the Dirichlet or Neumann boundary conditions in half balls
\begin{align}
\left \{ \begin{array}{rlr}
-{\rm div}\big(A(\frac{x}{\epsilon})\nabla u_\epsilon\big)=0 \quad \quad &\mbox{in} \ \mathbb B_{10}^+=\mathbb B_{10}\cap \{x_d\geq 0\},  \medskip \\
u_\epsilon=0 \ \mbox{or} \ \frac{\partial u_\epsilon}{\partial \nu}=0 \quad \quad &\mbox{on} \ \partial\mathbb B_{10}^+\cap \{x_d= 0\},
\end{array}
\right.
\label{mainhalf}
\end{align}
where
\begin{align}
A(y)=
\begin{pmatrix}
a_{11}(y) &\cdots & a_{1(d-1)}(y) & 0 \\
 \vdots  & \ddots  & \vdots & \vdots  \\
a_{(d-1)1}(y) &\cdots & a_{(d-1)(d-1)}(y) & 0 \\
0 & \cdots & 0 & a_{dd}(y)
\end{pmatrix},
\label{matrix}
\end{align}
 the conormal derivative $\frac{\partial u_\epsilon}{\partial \nu_\epsilon}$ is defined as
\begin{align*}
\frac{\partial u_\epsilon}{\partial \nu_\epsilon}=a_{ij}\frac{\partial u_\epsilon}{\partial x_j}n_i
\end{align*}
and $n=(n_1, \cdots, n_d)$ is a unit outer normal.
We are able to show the approximate three-ball inequality for solutions of (\ref{mainhalf}) in half balls.
\begin{theorem}
Let $u_\epsilon$ be a solution of (\ref{mainhalf}) with $A(y)$ satisfying (\ref{matrix}) in $\mathbb B^+_{10}$. For $0<R_1<R_2<\frac{\lambda R_3}{4}<\frac{\lambda}{4}$,
\begin{align}
\|u_\epsilon\|_{L^\infty(\mathbb B^+_{R_2})}\leq C\big\{\frac{ R_2}{R_3} \|u_\epsilon\|^\beta_{L^\infty(\mathbb B^+_{R_1})} \|u_\epsilon\|^{1-\beta}_{L^\infty(\mathbb B^+_{R_3})} +\frac{R^2_3}{R^2_1} [\epsilon \ln (\epsilon^{-1}+2)]^\beta \|u_\epsilon\|_{L^\infty(\mathbb B^+_{R_3})}\big\},
\label{halfspace}
\end{align}
where $\beta=\frac{\ln \frac{\lambda R_3}{2R_2}} {\ln \frac{R_3}{R_1} }$ and $C$ depends on $\lambda$ and $(\tau, \mu)$.
\end{theorem}

Similarly, we are able to consider an approximate three-ball inequality for eigenvalue type equations in half balls with the Dirichlet or Neumann boundary conditions
\begin{align}
\left \{ \begin{array}{rlr}
-{\rm div}\big(A(\frac{x}{\epsilon})\nabla u_\epsilon\big)=\lambda_{\epsilon, k} u_\epsilon \quad \quad &\mbox{in} \ \mathbb B_{10}^+=\mathbb B_{10}\cap \{x_d\geq 0\},  \medskip \\
u_\epsilon=0 \ \mbox{or} \ \frac{\partial u_\epsilon}{\partial \nu}=0 \quad \quad &\mbox{on} \ \partial\mathbb B_{10}^+\cap \{x_d= 0\},
\end{array}
\right.
\label{eigenhalf}
\end{align}
where $A(y)$ is in the form of (\ref{matrix}). The constants $\lambda_{\epsilon, k}$ are assumed to be positive and $\lambda_{\epsilon, k}\to\infty$ as $k\to \infty$. We can show the following corollary.

\begin{corollary}
Let $u_\epsilon$ be a solution of (\ref{eigenhalf}) with $A(y)$ satisfying (\ref{matrix}) in $\mathbb B^+_{10}$. For $0<R_1<R_2<\frac{\lambda R_3}{4}<\frac{\lambda}{4}$,
\begin{align}
\|u_\epsilon\|_{L^\infty(\mathbb B^+_{R_2})}&\leq C e^{2R_3 \sqrt{\lambda_{\epsilon, k}} }\big\{\frac{ R_2}{R_3} \|u_\epsilon\|^\beta_{L^\infty(\mathbb B^+_{R_1})} \|u_\epsilon\|^{1-\beta}_{L^\infty(\mathbb B^+_{R_3})} \nonumber \\ &+\frac{R^2_3}{R^2_1} [\epsilon \ln (\epsilon^{-1}+2)]^\beta \|u_\epsilon\|_{L^\infty(\mathbb B^+_{R_3})}\big\},
\label{lastcon}
\end{align}
where $\beta=\frac{\ln \frac{\lambda R_3}{2R_2}} {\ln \frac{R_3}{R_1} }$ and $C$ depends on $\lambda$ and $(\tau, \mu)$.
\label{corhalf}
\end{corollary}

The paper is organized as follows. Section 2 is devoted to some basic material on elliptic periodic homogenization. In section 3, we prove the Theorem 1, Theorem 2 and the Corollaries. The letter $C$ denotes a positive
constant that does not depend on $\epsilon$ or $u_\epsilon$. It may vary from line to line.

\section{Preliminaries}
This section introduces some background results on elliptic periodic homogenization. Most of the material can be found in books, e.g. \cite{BLP}, \cite{S}.
Let $ \chi(y)=(\chi_1(y), \cdots, \chi_d(y))\in H^1(\mathbb T^d; \mathbb R^d)$ denote the corrector of $\mathcal{L}_\epsilon$, where $\mathbb T^d=\mathbb R^d/ \mathbb Z^d$. It is known that $\chi(y)$ is the unique 1-periodic solution in $H^1(\mathbb T^d)$ satisfying
\begin{equation}
\left \{\begin{array}{lll}
\mathcal{L}_1(\chi_j)+\mathcal{L}_1(y_j)=0\quad \quad \mbox{in} \ \mathbb R^d,\medskip \\
\int_{\mathbb T^n  }\chi_j\, dy=0.
\end{array}
\right.
\end{equation}
Using the De Giorgi-Nash estimates, $\chi_j(y)$ is H\"older continuous. If $A(y)$ is H\"older continuous, then $\nabla \chi$ is bounded. See e.g. \cite{S}. Assume that $\Omega$ is  a bounded $C^{2, \eta}$ domain for some $\eta\in (0, 1)$ in $\mathbb R^d$ in the rest of the paper.
If $\{u_\epsilon\}$ is bounded in $H^1(\Omega)$, then any sequence $\{u_{\epsilon_l}\}$ with $\epsilon_l\to 0$ contains a subsequence that converges weakly in $H^1(\Omega)$. By the homogenization theory, one may conclude that as $\epsilon\to 0$,
\begin{equation}
\left\{ \begin{array}{lll}
A(\frac{x}{\epsilon})\nabla u_\epsilon\rightharpoonup \widehat{A}\nabla u \quad \quad & \mbox{weakly in } \  L^2(\Omega), \medskip\\
u_\epsilon\rightharpoonup u & \mbox{weakly in } \  L^2(\Omega), \medskip\\
u_\epsilon\to u &\mbox{stronlgy in }\ L^2(\Omega),
\end{array}
\right.
\end{equation}
where $\widehat{A}=(\widehat{a}_{ij})$ and
\begin{align}
\widehat{A}=\widehat{a}_{ij}=\int_{\mathbb T^n}\big[ a_{ij}+a_{ik}\frac{\partial{\chi_j}}{\partial y_k}\big]\,dy.
\end{align}
It is also true that the constant matrix $\widehat{a}_{ij}$ is symmetric and satisfies (\ref{ellip}) with the same parameter $\lambda$.
Thus,
the homogenized equation for $\mathcal{L}_\epsilon u_\epsilon=0$ is given by
\begin{align}\mathcal{L}_0 u=- {\rm div}(\widehat{A}\nabla u )=0.
\label{homoge}
\end{align}

Since $\widehat{A}$ is symmetric and positive definite, there exists a $d\times d$ matrix $S$ such that $S\widehat{A} S^T=I_{d\times d}$. Note that $\widehat{A}^{-1}=S^T S$ and
\begin{align}
\langle \widehat{A}^{-1}x, x\rangle= |Sx|^2.
\end{align}
We introduce a family of ellipsoids as
\begin{align}
E_r(\widehat{A})=\{ x\in \mathbb R^n: \langle (\widehat{A}^{-1} x, \ x \rangle\leq r^2\}.
\end{align}
By the assumption that $\widehat{A}$ satisfies (\ref{ellip}) as well, we can show that
\begin{align}
\mathbb B_{\sqrt{\lambda} r}(0)\subset E_r(\widehat{A}) \subset B_{\frac{r}{\sqrt{\lambda}}}(0).
\label{ellba}
\end{align}
We will write $E_r(\widehat{A})$ as $E_r$ if the context is understood.

The Dirichlet corrector is used to control the influence of the boundary data. The Dirichlet corrector $\Phi_{\epsilon, k}$ for the operator $\mathcal{L}_\epsilon$ in $\Omega$ is defined as
\begin{align}
\left \{ \begin{array}{ccl}
\mathcal{L}_\epsilon (\Psi_{\epsilon, k})&=0 \quad &\mbox{in} \ \Omega, \medskip \nonumber \\
\Psi_{\epsilon, k}&=P_k\quad &\mbox{on} \ \partial\Omega,
\end{array}
\right.
\end{align}
where $P_k= x_k$ for $1\leq k\leq d$. It was shown by Avellaneda and Lin in \cite{AL} that
\begin{align}
\| \nabla \Psi_{\epsilon, k}\|_{L^\infty(\Omega)}\leq C,
\label{resca}
\end{align}
where $C$ depends on $\lambda$, $(\tau, \mu)$ and $\Omega$. In the discussion in the next section, we need to know how the constant $C$ depends on $r$ in the ellipsoid $E_r$. By a rescaling argument, we can actually show that
\begin{align}
\| \nabla \Psi_{\epsilon, k}\|_{L^\infty(E_r)}\leq C
\label{corre}
\end{align}
with $C$ depends only on $\lambda$ and $(\tau, \mu)$.
For $x\in E_r$, let $y=\frac{x}{r}$. Then $y\in E_1$.  We perform a rescaling as
\begin{align}
\Psi_{\epsilon, k}(x)= r\tilde{\Psi}_{\epsilon, k}(\frac{x}{r})=r\tilde{\Psi}_{\epsilon, k}(y),
\end{align}
where $\tilde{\Psi}_{\epsilon, k}$ is the Dirichlet corrector for  the operator $\mathcal{L}_\epsilon$ in $E_1$. We can see that $\Psi_{\epsilon, k}(x)$ is the Dirichlet corrector for  the operator $\mathcal{L}_\epsilon$ in $E_r$, since
\begin{align}
\left \{ \begin{array}{rlr}
-{\rm div}( A(\frac{x}{\epsilon r})\nabla {\Psi}_{\epsilon, k}(x)=-{\rm div}( A(y)\nabla \tilde{\Psi}_{\epsilon, k}(y)=0 \quad &\mbox{in} \  E_r, \nonumber \medskip\\
{\Psi}_{\epsilon, k}=r\tilde{\Psi}_{\epsilon, k}(\frac{x}{r})= r\cdot \frac{x_k}{r}=x_k \quad &\mbox{on} \ \partial E_r.
\end{array}
\right.
\end{align}
It follows from (\ref{resca})  that
\begin{align*}
\| \nabla_y \tilde{\Psi}\|_{L^\infty(E_1)}\leq C
\end{align*}
with $C$ depending only on $\lambda$ and $(\tau, \mu)$.
Thus, it implies that
\begin{align}
\|\nabla_x {\Psi}_{\epsilon, k}\|_{L^\infty(E_r)}=\| \nabla_y \tilde{\Psi}\|_{L^\infty(E_1)}\leq C
\label{phi}
\end{align}
with the same constant $C$ in the last inequality. Therefore, we arrive at the estimates (\ref{corre}).

The asymptotic expansion of the Poisson kernel is well studied in homogenization theory. Let $P_\epsilon(x, \ y)$ denote the Poisson kernel for $\mathcal{L}_\epsilon$ in a bounded $C^{2, \eta}$ domain $\Omega$. It was shown by Kenig, Lin and Shen in \cite{KLS1} that we can write
\begin{align}
P_\epsilon(x, \ y)=P_0(x, \ y) w_\epsilon(y)+R_\epsilon(x,\ y),
\label{rewri}
\end{align}
where $P_0(x, \ y)$ is the Poisson kernel for the homogenized operator $\mathcal{L}_0$,
\begin{align}
w_\epsilon(y)= (\widehat{a}_{ij}n_i n_j )^{-1} \frac{\partial{\Psi_{\epsilon, k}}}{\partial n } n_k \cdot a_{ij}(\frac{y}{\epsilon}) n_i n_j,
\label{www}
\end{align}
$n=(n_1, \cdots, n_n)$ is a unit outer normal on $\partial \Omega$, and
\begin{align}
|R_\epsilon(x, y)|\leq \frac{ C\epsilon \ln |\epsilon^{-1}|x-y|+2|}{|x-y|^d} \quad \mbox{for any}\ x\in \Omega \ \mbox{and} \ y\in \partial\Omega.
\label{const}
\end{align}
The constant $C$ in (\ref{const}) depends on $\lambda$, $(\tau, \mu)$, and $\Omega$. From (\ref{ellip}), (\ref{phi}) and (\ref{www}), it follows that
\begin{align}
\|w_\epsilon(y)\|_{L^\infty(E_r)}\leq C
\label{www1}
\end{align}
with $C$ independent of $r$.
In the following discussions, we also need to know what the dependence of the constant $C$ in (\ref{const}) on the ellipsoid $E_r$ is. Note that for any $ x\in \Omega$ and $y\in \partial\Omega$,
\begin{align}
P_{\epsilon}(x, \ y)=-\frac{\partial G_\epsilon(x,\ y)}{\partial n(y)}  a_{ij}(\frac{y}{\epsilon}) n_i(y) n_j(y),
\end{align}
and
\begin{align}
P_0(x, \ y)=-\frac{\partial G(x,\ y)}{\partial n(y)}  \widehat{a}_{ij}(y) n_i(y) n_j(y),
\end{align}
where $G_\epsilon(x,\ y)$ and $G(x,\ y)$ are the Green functions for the operator $\mathcal{L}_\epsilon$ and $\mathcal{L}_0$, respectively.
 It follows from (\ref{rewri}) that
\begin{align}
R_\epsilon(x, \ y)&=-\frac{\partial G_\epsilon(x,\ y)}{\partial n(y)}  a_{ij}(\frac{y}{\epsilon}) n_i(y) n_j(y) \nonumber \\&+ \frac{\partial G(x,\ y)}{\partial n(y)}\frac{\partial \Psi_{\epsilon, k}}{\partial n_k(y)}n_k(y)  a_{ij}(\frac{y}{\epsilon}) n_i(y) n_j(y).
\end{align}
Thus, the estimate (\ref{const}) for $R_\epsilon(x, \ y)$ is shown in \cite{KLS1} as
\begin{align}
|R_\epsilon(x, \ y)|&\leq C | \frac{\partial G_\epsilon(x,\ y)}{\partial y_i}-\frac{\partial \Psi_{\epsilon ,k}}{\partial y_i} \  \frac{\partial G(x, y)}{\partial y_j} | \nonumber \\
&\leq C \frac{ \epsilon \ln |\epsilon^{-1}|x-y|+2|}{|x-y|^d}.
\label{howc}
\end{align}
To see how the $C$ in the last inequality depends on $r$ in $E_r$, we use the rescaling argument again. Let $x'=\frac{x}{r}$ and $y'=\frac{y}{r}$, where $x, y \in E_r$. Then
$x', y' \in E_1$.
 We do a rescaling as
\begin{align}
G_\epsilon (x, \ y)= r^{2-d} \tilde{G}_\epsilon(x', \ y'),
\end{align}
where  $\tilde{G}_\epsilon(x', \ y')$ is the Green function for the operator $\mathcal{L}_\epsilon$ in $E_1$.
Since
\begin{align}
\left \{ \begin{array}{rll}
-{\rm div}( a_{ij}(\frac{x}{\epsilon r})\nabla G_{\epsilon}(x, \ y) )=-r^{-d}{\rm div}( a_{ij}(\frac{x'}{\epsilon})\nabla \tilde{G}_{\epsilon}(y'))= r^{-d} \delta(x'-y')=\delta(x-y) \quad &\mbox{in} \  E_r, \nonumber \medskip\\
G_{\epsilon}(x, \ y)=r^{2-d} \tilde{G}_{\epsilon}(x', \ y')=0 \quad &\mbox{on} \ \partial E_r,
\end{array}
\right.
\end{align}
 it implies that $G_\epsilon (x, \ y)$ is the Green function for the operator $\mathcal{L}_\epsilon$ in $E_r$.
From (\ref{howc}), we have
\begin{align}
| \frac{\partial {G}_\epsilon(x,\ y)}{\partial y_i}-\frac{\partial \Psi_{\epsilon ,k}}{\partial y_i} \  \frac{\partial G(x, y)}{\partial y_j} |&= r^{1-d} | \frac{\partial G_\epsilon(x',\ y')}{\partial y'_i}-\frac{\partial \Psi_{\epsilon ,k}}{\partial y'_i} \  \frac{\partial G(x', y')}{\partial y'_j} | \nonumber \\
&\leq \frac{ C r^{1-d} \epsilon \ln [\epsilon^{-1}|x'-y'|+2 ] }{|x'-y'|^d} \nonumber \\
&=\frac{ C r \epsilon \ln [\epsilon^{-1}\frac{|x-y|}{r}+2 ] }{|x-y|^d},
\end{align}
where $C$ in the last inequality depends only on $\lambda, (\tau, \mu)$ since $x', y'\in E_1$. This implies that
\begin{align}
|R_\epsilon(x, \ y)|\leq \frac{ C r \epsilon \ln [\epsilon^{-1}\frac{|x-y|}{r}+2 ] }{|x-y|^d}
\label{RRR}
\end{align}
for $x\in E_r$ and $y\in \partial E_r$ and $C$ is independent of $r$.

\section{Approximate three-ball inequality }
This section is devoted to the proof of three-ball inequality (\ref{result}) in Theorem 1. We adapt the method of using the Poisson kernel in \cite{GM}, for harmonic functions, to elliptic homogenization. We want to have an explicit form for the Poisson kernel of the homogenized operator $\mathcal{L}_0$. We transform the operator $\mathcal{L}_0$ into the classical Laplacian operator.
Let $u(x)=w(Sx)$. Then
\begin{align}
0={\rm div}(\widehat{A}\nabla u)={\rm div} (S\widehat{A} S^T \nabla w)=\triangle w.
\end{align}
We can represent the harmonic function $w(x)$ by Poisson kernel as
\begin{align}
w(x)=\int_{S_r} \gamma_d \frac{ r^2-|x|^2}{r|x-y|^d}w(y)\,dy,
\label{harmo}
\end{align}
where $S_r$ is the $d-1$ dimensional sphere centered at the origin with radius $r$ and $\gamma_d=2\pi^{\frac{d}{2}}/\Gamma(\frac{d}{2})$ is the surface area of $S_1$ with $d\geq 2$. By the relation of $u(x)=w(Sx)$, we can transform (\ref{harmo}) to the solution $u$ for homogenized equation (\ref{homoge}). It follows that
\begin{align}
u(x)=\int_{\partial E_r} P_0(x, y) u(y)\,dy,
\end{align}
where
\begin{align}
P_0(x, y)=\gamma_d |S| \frac{ r^2-|Sx|^2}{ r|Sx-Sy|^d}.
\label{poiss}
\end{align}

Following \cite{GM}, we are going to apply the Lagrange interpolation method to obtain the three-ball inequality (\ref{result}).
Let us briefly review the standard Lagrange interpolation method in numerical analysis \cite{DB}.
Set
\begin{align}
\Phi_m(z)=(z-u_1)(z-u_2)\cdots (z-u_m)
\label{defin}
\end{align}
for $z, u_j\in \mathcal{C}$ with $j=1, \cdots, m$. Let $\mathcal{D}$ be a simply connected open domain in the complex plane $\mathcal{C}$ that contains the nodes $\tilde{u}, u_1, \cdots, u_m$. Assume that $f$ is an analytic function without poles in the closure of $\mathcal{D}$. By well-known calculations, it holds that
\begin{align}
\frac{1}{z-\tilde{u}}=\sum^m_{j=1} \frac{\Phi_{j-1}(\tilde{u})}{\Phi_j(z)}+\frac{\Phi_m(\tilde{u})}{(z-\tilde{u})\Phi_m(z)}.
\end{align}
Multiplying the last identity by $\frac{1}{2\pi i} f(z)$ and integrating along the boundary of $\mathcal{D}$ leads to
\begin{align}
\frac{1}{2\pi i}\int_{\partial \mathcal{D}}\frac{f(z)}{z-\tilde{u}}\, dz= \sum^m_{j=1} \frac{\Phi_{j-1}(\tilde{u})}{2\pi i}\int_{\partial \mathcal{D}} \frac{f(z)}{\Phi_j(z)}\, dz+(R_m f)(\tilde{u}),
\end{align}
where
\begin{align*}
(R_m f)({\tilde{u}})= \frac{1}{2\pi i}\int_{\partial \mathcal{D}}\frac{\Phi_m(\tilde{u})f(z) }{(z-\tilde{u})\Phi_m(z)}\, dz.
\end{align*}
By the residue theorem, we obtain that
\begin{align}
(R_m f)(\tilde{u})&=\sum^m_{j=1} \frac{\Phi_m(\tilde{u})}{(u_j-\tilde{u})\Phi'_m(u_j)} f(u_j)+ f(\tilde{u}) \nonumber \\
&=-\sum^m_{j=1}\prod^m_{i\not=j} \frac{\tilde{u}-u_i}{u_j-u_i} f(u_j) +f(\tilde{u}),
\label{lang}
\end{align}
where $(R_m f)(\tilde{u})$ is called the interpolation error. Based on the identity (\ref{lang}), the idea is to approximate $f(\tilde{u})$ by a linear combination of the polynomials $-\sum^m_{j=1}\prod^m_{i\not=j} \frac{\tilde{u}-u_i}{u_j-u_i} f(u_j)$, and then control the error term $(R_m f)(\tilde{u})$ efficiently. See chapter 4 in \cite{DB} for more information.

With these preliminaries, we are able to present the proof of Theorem 1.

\begin{proof}[Proof of Theorem 1]
In order to obtain the approximate three-ball theorem for the solution in (\ref{main}) in elliptic periodic homogenization, we consider the Lagrange interpolation for $f(t)=P_0(tx_0\frac{r_1}{r_2}, y)$, where $0<r_1 <r_2<\frac{R}{4}<1$. We fix a point $x_0$ such that $ \sqrt{\langle\widehat{A}^{-1}x_0, x_0\rangle}=|Sx_0|\leq r_2$. We approximate $P_0(x_0, y)$ by a linear combination of the form $\sum^m_{i=1} c_i P_0(x_i, y)$ with $|S x_i|\leq r_1$. Then we need to estimate the sum of the absolute values of the coefficients $c_i$ in the linear combination and the error $(R_m P_0)(x_0, \ y)$ of the approximation.

 We choose points $x_i=t_i x_0\frac{r_1}{r_2}$ on the segment $[0, x_0\frac{r_1}{r_2}]$ with $t_i\in (0, \ 1).$
We select $u_i= t_i$ in the definition of $\Phi_m$ in (\ref{defin}) and $\tilde{u}=\frac{r_2}{r_1}$. Define
\begin{align}
c_j=\prod^m_{i\not =j} \frac{r_2r_1^{-1}-t_i}{ t_j- t_i}.
\end{align}
Since $0<t_i<1$, direct calculations show that
\begin{align}
|c_j|\leq \frac{(r_2r_1^{-1})^m}{|\Phi'_m(t_j)|}.
\end{align}
To further estimate $|c_j|$, we choose $t_i$ to be the Chebyshev nodes, i.e. $t_i=\cos (\frac{(2i-1)\pi}{2m})$.
Then we can write
\begin{align*}\Phi_m(t)=2^{1-m} T_m(t),\end{align*}
 where $T_m$ is the Chebyshev polynomial of the first kind. We also know that
\begin{align}
\Phi'_m(t)=m2^{1-m} U_{m-1}(t),
\label{cheby}
\end{align}
where $U_{m-1}$ is the Chebyshev polynomial of the second kind. See e.g. section 3.2.3 in \cite{DB}. At each $t_i$, we have
\begin{align*}
U_{m-1}(t_i)= U_{m-1} \big( \cos \frac{(2i-1)\pi}{2m}\big)=\frac{\sin \frac{(2i-1)\pi}{2}}{\sin \frac{(2i-1)\pi}{2m}}=\frac{(-1)^{i-1}}{\sin \frac{(2i-1)\pi}{2m} }.
\end{align*}
By (\ref{cheby}), it follows that
\begin{align*}
|\Phi'_m(t_i)|\geq m 2^{1-m}.
\end{align*}
Therefore, we can show that
\begin{align}
|c_j|\leq (2 m)^{-1} (\frac{2r_2}{r_1})^m.
\label{ccc}
\end{align}

To estimate the error of the approximation $(R_m P_0)(x_0, \ y)$, we do an analytic extension of the function $f(t)=P_0( t x_0 \frac{r_1}{r_2}, y)$ to the disc of radius $\frac{R}{2r_1}$ centered at the origin in the complex plane $\mathcal{C}$. Note that $P_0(x,\ y)$ is independent of $\epsilon$. By the explicit form of $P_0(x,\ y)$ in (\ref{poiss}) for $\partial E_R$, we have
\begin{align}
f(z)=\gamma_d |S| \frac{R^2-|Sx_0|^2 r_1^2 r^{-2}_2 z^2}{R | r_1 r_2^{-1} zSx_0-Sy|^d}.
\end{align}
If $R$ is fixed, then
\begin{align}
|f(z)|\leq CR^{-(d-1)}
\end{align}
in the disc. Hence, $f(z)$ is bounded. By (\ref{defin}), we can also see that
\begin{align}
|\Phi_m(z)|\geq \big( (\frac{R}{2r_1})-1\big)^m \quad \mbox{on the circle} \ |z|=\frac{R}{2r_1}.
\label{lower}
\end{align}
and \begin{align}
|\Phi_m(\frac{r_2}{r_1})|\leq (\frac{r_2}{r_1})^m.
\label{upper}\end{align} By the identity (\ref{lang}), we estimate the interpolation error as follows
\begin{align}
|(R_m P_0)(x_0, \ y)|&=|P_0(x_0, \ y)-\sum^m_{i=1} c_i P_0(x_i, \ y)| \nonumber \\
&=|f(\frac{r_2}{r_1})-\sum^m_{i=1} c_i f(t_i)|\nonumber \\
&=|\frac{1}{2\pi i}\int_{|z|=\frac{R}{2r_1}}\frac{\Phi_m(r_2 r_1^{-1})f(z) }{(z-r_2 r_1^{-1})\Phi_m(z)}\, dz|\nonumber \\
&\leq C \frac{2^m r_2^m}{R^{d-2} (R-2r_1)^m (R-2r_2)} \nonumber \\
&\leq C \frac{2^m r_2^m}{R^{d+m-1}},
\label{possi}
\end{align}
where we have used estimates (\ref{lower}), (\ref{upper}), the assumption that $0<r_1<r_2<\frac{R}{4}<R$ in the last inequality, and the constant $C$ in the last inequality does not depend on $m$.
Note that (\ref{ccc}) yields
\begin{align}
\sum^m_{j=1}|c_j|\leq 2^{-1} (\frac{2r_2}{r_1})^m.
\label{summ}
\end{align}

With these estimates for the Poisson kernel, we are going to estimate the supremum $u_\epsilon$ in ellipsoids.
We can write the solutions of (\ref{main}) in $E_R$ as
\begin{align}
u_\epsilon(x)=\int_{\partial E_R} P_\epsilon (x, \ y) u_\epsilon(y)\,dy.
\end{align}
Thanks to (\ref{rewri}), we can split the last integral as
\begin{align}
| \int_{\partial E_R} P_\epsilon (x, \ y) u_\epsilon(y)\,dy|&\leq |\int_{\partial E_R} \sum^m_{i=1} c_i P_\epsilon(x_i, \ y) u_\epsilon(y)\, dy| \nonumber \\& + |\int_{\partial E_R}\big(P_\epsilon(x, \ y)- \sum^m_{i=1} c_i P_\epsilon(x_i, \ y)\big)  u_\epsilon(y)\, dy| \nonumber \\
&\leq \sum^m_{i=1}|c_i||u_\epsilon(x_i)|+\int_{\partial E_R}|P_0(x, \ y)- \sum^m_{i=1} c_i P_0(x_i, \ y)||w_\epsilon(y)||u_\epsilon(y)|\, dy\nonumber \\
&+\int_{\partial E_R}|R_\epsilon(x, \ y)- \sum^m_{i=1} c_i R_\epsilon(x_i, \ y)||u_\epsilon(y)|\, dy.
\end{align}
Let $x=x_0$. By the relations of $x_i=t_i x_0 \frac{r_1}{r_2}$ and $|Sx_0|\leq r_2$, then all $x_i\in E_{r_1}$. Taking into account the estimates (\ref{www1}), (\ref{RRR}),  (\ref{possi}) and (\ref{summ}) gives that
\begin{align}
|u_\epsilon(x_0)|\leq (\frac{2r_2}{r_1})^m \sup_{E_{r_1}}|u_\epsilon|+ C \frac{2^m r_2^m}{R^{m}}\sup_{E_{R}}|u_\epsilon|+ C\epsilon \ln (\epsilon^{-1}+2)(\frac{2r_2}{r_1})^m \sup_{E_{R}}|u_\epsilon|,
\end{align}
where those $C$ do not depend on $m$, $r_1$, $r_2$, or $R$.  Since $x_0$ is an arbitrary point in $E_{r_2}$, it follows that
\begin{align}
\sup_{E_{r_2}}|u_\epsilon|\leq C\big\{(\frac{2r_2}{r_1})^m \sup_{E_{r_1}}|u_\epsilon|+ (\frac{2 r_2}{R})^m\sup_{E_{R}}|u_\epsilon|+ \epsilon \ln (\epsilon^{-1}+2)(\frac{2r_2}{r_1})^m \sup_{E_R}|u_\epsilon|      \big\}.
\label{mini}
\end{align}

Next we aim to minimize the summation of the terms in the right hand side of (\ref{mini}) by choosing the integer value $m$. For ease of notations,
let
\begin{align}
\sup_{E_{r_1}}|u_\epsilon|=\delta, \quad \sup_{E_{R}}|u_\epsilon|=M.
\end{align}
First of all, we choose a value of $m$ such that
\begin{align}
(\frac{2r_2}{r_1})^m \delta=\frac{(2r_2)^m}{R^{m}}M.
\end{align}
Solving the equality gives that
\begin{align*}
m=\frac{\ln M/\delta}{\ln R/r_1}.
\end{align*}
We  define an integer value $$m_0=\lfloor\frac{\ln M/\delta}{\ln R/r_1}\rfloor+1,$$ where $\lfloor \cdot \rfloor$ denotes its integer part.  We split the discussion of the minimization into  two cases. \medskip \\
\textbf{Case 1}: The case $\epsilon \ln (\epsilon^{-1}+2)(\frac{2r_2}{r_1})^{m_0} \leq   (\frac{2r_2}{R})^{m_0} $. \medskip \\
In this case, let $m=m_0$ in the estimate (\ref{mini}). The third term can be incorporated into the second term in the right hand side of (\ref{mini}). It follows that
\begin{align}
\sup_{E_{r_2}}|u_\epsilon|&\leq C\big\{(\frac{2r_2}{r_1})^{m_0} \delta+ (\frac{2 r_2}{R})^{m_0} M   \big\} \nonumber \\
&\leq C (\frac{2 r_2}{R}) M^{1-\alpha} \delta^\alpha,
\label{case1}
\end{align}
where \begin{align}\alpha= \frac{\ln \frac{R}{2r_2}} {\ln \frac{R}{r_1} }.
\label{alpha}\end{align} \medskip \\

\textbf{Case 2}: The case $\epsilon \ln (\epsilon^{-1}+2)(\frac{2r_2}{r_1})^{m_0} >(\frac{2r_2}{R})^{m_0} $. \medskip \\
In this case,  from the definition of $m_0$, we see that
\begin{align}
\epsilon \ln (\epsilon^{-1}+2)> \frac{\delta r_1}{MR}.
\label{value}
\end{align}
That is,
\begin{align}
\sup_{E_{r_1}}|u_\epsilon|\leq \epsilon \ln (\epsilon^{-1}+2)\frac{R}{r_1} \sup_{E_{R}}|u_\epsilon|.
\label{simp}
\end{align}
In order to obtain the minimum for the terms in the right hand sides of (\ref{mini}), we choose the value of $\hat{m}$ such that
\begin{align}
\epsilon \ln (\epsilon^{-1}+2)(\frac{2r_2}{r_1})^{\hat{m}} =(\frac{2r_2}{R})^{\hat{m}}.
\end{align}
Solving the equality yields that
\begin{align}
\hat{m}=\frac{ \ln \big( \epsilon \ln (\epsilon^{-1}+2)\big)}{\ln \frac{r_1}{R}}.
\end{align}
We define the integer value as
\begin{align}
m_1=\lfloor\frac{ \ln \big( \epsilon \ln (\epsilon^{-1}+2)\big)}{\ln \frac{r_1}{R}}\rfloor+1.
\end{align}
Substituting $m=m_1$ in the inequality (\ref{mini}) again, we can incorporate the second term into the third term in the right hand side of (\ref{mini}). Taking (\ref{simp}) into account, we obtain that
\begin{align}
\sup_{E_{r_2}}|u_\epsilon|&\leq C\big\{ (\frac{2r_2}{r_1})^{m_1}\epsilon \ln (\epsilon^{-1}+2)\frac{R}{r_1} M+  \epsilon \ln (\epsilon^{-1}+2)                       (\frac{2r_2}{r_1})^{m_1} M \big \} \nonumber \\
&\leq C \frac{2r_2 R}{r_1^2}  \exp\{ \frac{\ln\frac{2r_2}{r_1} \ln [\epsilon \ln (\epsilon^{-1}+2)]} { \ln \frac{r_1}{R}} \}\epsilon \ln (\epsilon^{-1}+2)M  \nonumber \\
&\leq C  \frac{R^2 }{r^2_1}  [\epsilon \ln (\epsilon^{-1}+2)]^{\frac{\ln \frac{R}{2r_2}} {\ln \frac{R}{r_1} }} M \nonumber \\
&= C  \frac{R^2 }{r^2_1}  [\epsilon \ln (\epsilon^{-1}+2)]^{\alpha} M
\label{case2}
\end{align}
with \begin{align}\alpha= \frac{\ln \frac{R}{2r_2}} {\ln \frac{R}{r_1} }.
\end{align}
Combining the estimates in (\ref{case1}) and (\ref{case2}) in these two cases, we arrive at
\begin{align}
\sup_{E_{r_2}}|u_\epsilon|\leq C \big \{\frac{r_2}{R} (\sup_{E_{r_1}}|u_\epsilon|)^\alpha (\sup_{E_{R}}|u_\epsilon|)^{1-\alpha}+ \frac{R^2}{r^2_1}[\epsilon \ln (\epsilon^{-1}+2)]^\alpha \sup_{E_{R}}|u_\epsilon|\big\}
\label{soids}
\end{align}
in three different sizes of ellipsoids. By the assumption of $r_1$, $r_2$ and $R$, we also see that $0<\alpha<1$. This completes the proof of Theorem 1.
\end{proof}
Now we are ready to show the proof of Corollary 1.

\begin{proof}[Proof of Corollary 1]

We make use of the approximate three-ball inequality in ellipsoids in (\ref{soids}). We choose  $0<r_1\leq \frac{r_1}{\lambda}<r_2<\frac{R}{4}<\frac{R}{4\lambda}$ for $R<\sqrt{\lambda}$. By the relation of ellipsoids and balls in (\ref{ellba}), we have
\begin{align}
\sup_{\mathbb B_{\sqrt{\lambda} r_2}}|u_\epsilon|\leq C \big \{\frac{r_2}{R} (\sup_{\mathbb B_{r_1/\sqrt{\lambda}}}|u_\epsilon|)^\alpha (\sup_{\mathbb B_{R/\sqrt{\lambda}}}|u_\epsilon|)^{1-\alpha}+ \frac{R^2}{r^2_1}[\epsilon \ln (\epsilon^{-1}+2)]^\alpha \sup_{\mathbb B_{R/\sqrt{\lambda}}}|u_\epsilon|\big\}.
\end{align}

Let $R_1=r_1/\sqrt{\lambda}$, $R_2=\sqrt{\lambda} r_2$ and $R_3=R/\sqrt{\lambda}$, Then $0<R_1<R_2<\frac{R_3}{4}<\frac{1}{4}$. The last inequality implies that
\begin{align}
\sup_{\mathbb B_{R_2}}|u_\epsilon|\leq C \big \{\frac{R_2}{R_3} (\sup_{\mathbb B_{R_1}}|u_\epsilon|)^\beta (\sup_{\mathbb B_{R_3}}|u_\epsilon|)^{1-\beta}+ \frac{R_3^2}{R^2_1}[\epsilon \ln (\epsilon^{-1}+2)]^\beta \sup_{\mathbb B_{R_3}}|u_\epsilon|\big\},
\label{newfind}
\end{align}
where \begin{align} \beta= \frac{\ln \frac{\lambda R_3}{2R_2}} {\ln \frac{R_3}{R_1}}\end{align} is derived from the $\alpha$ value in (\ref{alpha}) and $C$ depends only on $\lambda$ and $(\tau, \mu)$.
It is easy to see that $\beta<1$. By the value of $\lambda$, it is possible that $\beta<0$. However, the second term in the right hand side of (\ref{newfind}) implies such inequality holds trivially when $\epsilon$ is sufficiently small.
 To have $\beta>0$, we need to have $R_2<\frac{\lambda R_3}{2}$. Together with $R_2<\frac{R_3}{4}<\frac{1}{4}$,
 we choose $0<R_1<R_2<\frac{\lambda R_3}{4}<\frac{\lambda}{4}$. Thus, the inequality (\ref{newfind}) holds with $0<\beta<1$.
This completes the proof of Corollary.

\end{proof}
From Corollary 1,  we can give the proof of Corollary 3.
\begin{proof}[Proof of Corollary 3] We introduce a new function $v_{\epsilon}(x, t)$ as
\begin{align*}
v_{\epsilon}(x, t)= e^{\sqrt{\lambda_{\epsilon, k}} t  }u_\epsilon(x) \quad \quad \mbox{in}  \ \mathbb B_{10}\times (-10, \ 10).
\end{align*}
From eigenvalue type equation (\ref{eigen}), the new function $v_{\epsilon}(x, t)$ solves the equation
\begin{align}
\mathcal{L}_\epsilon v_\epsilon- \partial^2_t v_\epsilon=0 \quad \quad \mbox{in}  \ \mathbb B_{10}\times (-10, \ 10).
\label{trans}
\end{align}
This new homogenization equation (\ref{trans}) has the coefficient matrix
$$\begin{pmatrix}
(a_{ij})_{d\times d}  & 0 \medskip\\
0 & 1
\end{pmatrix}$$
which also satisfies conditions (\ref{ellip}), (\ref{perio}) and (\ref{holder}).
Let $\tilde{\mathbb B}_r$ be the ball with radius $r$ in $\mathbb B_{10}\times (-10, \ 10)$. We may write $\tilde{\mathbb B}_r$ as $\mathbb B_r \times (-r, \ r)$.
From Corollary 1, it holds that
\begin{align}
\|v_\epsilon\|_{L^\infty(\tilde{\mathbb B}_{R_2})}\leq C\big\{\frac{ R_2}{R_3} \|v_\epsilon\|^\beta_{L^\infty(\tilde{\mathbb B}_{R_1})} \|v_\epsilon\|^{1-\beta}_{L^\infty(\tilde{\mathbb B}_{R_3})} +\frac{R^2_3}{R^2_1} [\epsilon \ln (\epsilon^{-1}+2)]^\beta \|v_\epsilon\|_{L^\infty(\tilde{\mathbb B}_{R_3})}\big\}
\end{align}
for  $R_1, R_2, {R_3}, \beta$ given in Corollary 1.
By the definition of $v_\epsilon(x, t)$ and $\tilde{\mathbb B}_r$, we obtain that
\begin{align}
e^{-\sqrt{\lambda_{\epsilon, k}} R_2 }\|u_\epsilon\|_{L^\infty({\mathbb B}_{R_2})}&\leq C
\big\{  \frac{ R_2}{R_3} e^{ \sqrt{\lambda_{\epsilon, k}}R_3}\|u_\epsilon\|^\beta_{L^\infty({\mathbb B}_{R_1})} \|u_\epsilon\|^{1-\beta}_{L^\infty({\mathbb B}_{R_3})} \nonumber  \\ &+ \frac{R^2_3}{R^2_1} e^{ \sqrt{\lambda_{\epsilon, k}}R_3 }[\epsilon \ln (\epsilon^{-1}+2)]^\beta \|u_\epsilon\|_{L^\infty({\mathbb B}_{R_3})}\big\}.
\end{align}
Thus, the corollary follows.
\end{proof}

Let us show another consequence of the three-ball inequality in Corollary 1, that is, the approximate propagation of smallness estimates.

\begin{proof}[Proof of Corollary 4] As a consequence of De Giorgi-Nash estimates, the solution $u_\epsilon$ in (\ref{main}) is a continuous function.
Hence there exists $\bar x\in \mathbb B_9$ such that $u_\epsilon(\bar x)=\sup_{\mathbb B_9}|u_\epsilon(x)|$. We select a sequence of balls with radius $2r$ centered at $x_1=0, \cdots, x_m$ so that $x_{i+1}\in \mathbb B_r(x_i)$, $\mathbb B_r(x_{i+1})\subset \mathbb B_{2r}(x_i)$ and  $\bar x\in \mathbb B_{2r}(x_m)$. The number $m$ depends on $r$. By the way we choose $x_{i+1}$, it holds that
\begin{align}
\|u_\epsilon\|_{L^\infty( \mathbb B_r(x_{i+1}))}\leq \|u_\epsilon\|_{L^\infty( \mathbb B_{2r}(x_i))}.
\label{infnorm}
\end{align}
Applying the three-ball inequality (\ref{result2}) for $x_1=0$ and using that $\|u_\epsilon\|_{L^\infty(\mathbb B_{10})}\leq 1$, we have
\begin{align}
\|u_\epsilon\|_{L^\infty(\mathbb B_{2r}(x_1))}\leq C\{\delta^\beta+ [\epsilon\ln (\epsilon^{-1}+2)]^\beta\}
\label{again}
\end{align}
 with $\beta=\frac{\ln \frac{9}{4}} {\ln \frac{9}{\lambda}}$. Then we choose $x_2\in \mathbb B_{2r}(x_1)$ such that $\mathbb B_r({x_2})\subset \mathbb B_{2r}(x_1)$. Thus,
\begin{align}
\|u_\epsilon\|_{L^\infty( \mathbb B_r(x_2))}\leq \|u_\epsilon\|_{L^\infty( \mathbb B_{2r}(x_1))}.
\label{using}
\end{align}
Applying  the three-ball inequality (\ref{result2}) for balls centered at $x_2$ and using (\ref{again}) and (\ref{using}), we have
\begin{align}
\|u_\epsilon\|_{L^\infty( \mathbb B_{2r}(x_2))}&\leq C \{C\delta^\beta+ C[\epsilon\ln (\epsilon^{-1}+2)]^\beta\}^\beta+ C [\epsilon\ln (\epsilon^{-1}+2)]^\beta \nonumber \\
&=C^{\beta+1} \delta^{\beta^2}+ C^{\beta+1} [\epsilon\ln (\epsilon^{-1}+2)]^{\beta^2}+ C[\epsilon\ln (\epsilon^{-1}+2)]^\beta.
\end{align}

Iterating this argument with three-ball inequality at points $x_3, \cdots,$ up to $x_m$ and using (\ref{infnorm}) and the fact that
$\|u_\epsilon\|_{L^\infty(\mathbb B_{10})}\leq 1$, we obtain that
\begin{align}
\|u_\epsilon\|_{L^\infty( \mathbb B_{2r}(x_m))}\leq C^{\sum^m_{i=1}\beta^{m-i}} \delta^{\beta^m}+\sum^m_{j=1} C^{\sum^j_{i=1} \beta^{i-1}} [\epsilon \ln(\epsilon^{-1}+2)]^{\beta^j} \nonumber
\end{align}
Since $0<\beta<1$, we have
\begin{align}
\|u_\epsilon\|_{L^\infty( \mathbb B_{2r}(x_m))} \leq C^{\frac{1}{1-\beta}}\delta^{\beta^m}+ m C^{\frac{1}{1-\beta}}[\epsilon \ln(\epsilon^{-1}+2)]^{\beta^m}.
\end{align}
If $r$ is fixed, then $m$ is a fixed number. Thus,
\begin{align}
\|u_\epsilon\|_{L^\infty(\mathbb B_{9})}
\leq C^{\frac{1}{1-\beta}}\delta^{\beta^m}+ m C^{\frac{1}{1-\beta}}[\epsilon \ln(\epsilon^{-1}+2)]^{\beta^m}
\end{align}
for $\epsilon \ln(\epsilon^{-1}+2)\leq 1$.
\end{proof}

Compared with the previous approximate propagation of smallness result, we are able to show a full  propagation of smallness result with aid of the doubling inequality in Theorem A.
\begin{proof}[Proof of Corollary \ref{cornew}]
Let $r_k= 2^k r_0$. Choose $k_0$ such that $2^{k_0} r_0=\sqrt{\lambda}$. Then $k_0\approx \log_2(\frac{\sqrt{\lambda}}{r_0})$. Let $M_k=\|u_\epsilon\|_{L^\infty(\mathbb B_{r_k})}$ so that $M_0=\|u_\epsilon\|_{L^\infty(\mathbb B_{r_0})}$ and $M_{k_0}=\|u_\epsilon\|_{L^\infty(\mathbb B_{\sqrt{\lambda}})}$. Choosing
\begin{align}
N=\frac{\|u_\epsilon\|_{L^\infty(\mathbb B_{2/\sqrt{\lambda}})}} {\|u_\epsilon\|_{L^\infty(\mathbb B_{\sqrt{\lambda}})}},
\label{newN}
\end{align} from Theorem A, we have
\begin{align}
M_0=\frac{M_0}{M_1}\cdot \frac{M_1}{M_2}\cdot \frac{M_2}{M_3}\cdot \frac{M_{k_0-1}}{M_{k_0}}\cdot M_{k_0}\geq (\frac{1}{C(N)})^{k_0} M_{k_0}.
\end{align}
It also follows that
\begin{align}
M_{k_0}\leq M_0 C(N)^{k_0 }.
\label{itera}
\end{align}
From the $N$ value in (\ref{newN}), the boundedness assumption of $u_\epsilon$ and the definition of $M_{k_0}$, it follows that $N\leq \frac{1}{M_{k_0}}$. By (\ref{itera}) and monotonicity of $C(N)$, we arrive at
\begin{align}
M_{k_0}\leq M_0 C(\frac{1}{M_{k_0}})^{k_0 }.
\label{mkmk}
\end{align}
Suppose that $M_{k_0}>\eta_0$ and $M_0=\|u_\epsilon\|_{L^\infty(\mathbb B_{r_0})}\leq \delta_0$, from (\ref{mkmk}),
\begin{align}
M_{k_0}\leq \delta_0  C(\frac{1}{M_{k_0}})^{k_0 }\leq \delta_0  C(\frac{1}{\eta_0})^{k_0 }.
\end{align}
If we choose $\delta_0$ such that $\delta_0C(\frac{1}{\eta_0})^{k_0 }\leq \frac{\eta_0}{2}$, this gives that $M_{k_0}\leq \frac{\eta_0}{2}$, which contradicts the assumption $M_{k_0}>\eta_0$. Hence, if $\delta_0\leq \frac{\eta_0}{2}({ C(\frac{1}{\eta_0}) })^{-k_0}$, we must have $M_{k_0}\leq \eta_0$ as desired. This completes the corollary. Notice that the dependence of $\delta_0$ is not explicit.
\end{proof}

As a consequence of Corollary 1, we can give the proof of Theorem 2.
\begin{proof}[Proof of Theorem 2] To use the conclusion in Corollary 1, we extend the  equation (\ref{mainhalf}) in half balls to balls.
We do an even extension for the metric $A(y)$ across the half space $\{x| x_d=0\}$. Then we define a new coefficient matrix as
\begin{align}
\tilde{A}(\frac{x}{\epsilon})=\left\{ \begin{array}{llr}
A(\frac{x'}{\epsilon}, \ \frac{x_d}{\epsilon} ) \quad \quad &\mbox{for} \ x_d\geq 0, \nonumber \medskip \\
A(\frac{x'}{\epsilon}, \ \frac{-x_d}{\epsilon} ) \quad \quad &\mbox{for} \ x_d< 0,
\end{array}
\right.
\end{align}
where $x'=(x_1, \cdots, x_{d-1})$. We can see that $\tilde{A}(y)$ still satisfies the ellipticity condition (\ref{ellip}) and 1-periodic condition (\ref{perio}). To verify that $\tilde{A}(y)$ also satisfies the H\"older continuity (\ref{holder}), we just need to check the H\"older continuity for $\tilde{A}(y)$ in $y_d$ direction. Let $y_d^1>0$ and $y_d^2<0$. By the condition (\ref{holder}) for ${A}(y)$, we have
\begin{align}
|\tilde{A}(y', y_d^1)- \tilde{A}(y', y_d^2)|&=|{A}(y', y_d^1)-{A}(y', 0)+{A}(y', 0)- {A}(y', -y_d^2)| \nonumber \\
&\leq \tau |y_d^1|^\mu+\tau |y_d^2|^\mu \nonumber \\
&\leq 2\tau |y_d^1-y_d^2|^\mu.
\end{align}
This implies that H\"older continuity as (\ref{holder}) holds for $\tilde{A}(y)$.
For the Dirichlet boundary conditions, we do an odd extension for $u_\epsilon(x)$ across the half space $\{x|x_d=0\}$.
Then we have a new function $v_\epsilon(x)$ defined as follows,
\begin{align}
v_\epsilon(x)=\left\{ \begin{array}{llr}
u_\epsilon(x', \  x_d) \quad \quad &\mbox{for} \ x_d\geq 0, \nonumber \medskip \\
-u_\epsilon(x', \  -x_d)  \quad \quad &\mbox{for} \ x_d< 0.
\end{array}
\right.
\end{align}

Since the solution $u_\epsilon(x)$ in (\ref{mainhalf}) for the Dirichlet boundary conditions is Lipschitz continuous by the homogenization theory, e.g. \cite{AL}, the new $v_\epsilon(x)$ is Lipschitz continuous and is in $H^1(\mathbb B_{10})$. We claim that $v_\epsilon(x)$ satisfies the elliptic homogenization
\begin{align}
-{\rm div}\big(\tilde{A}(\frac{x}{\epsilon})\nabla v_\epsilon\big)=0 \quad \quad \mbox{in} \ \mathbb B_{10}.
\label{newequ}
\end{align}

To verify that $v_\epsilon(x)$ is a weak solution of  (\ref{newequ}), we need to check that the following holds,
\begin{align}
\int_{\mathbb B_{10}} \tilde{A}(\frac{x}{\epsilon}) \nabla v_\epsilon\cdot \nabla \phi(x) \, dx=0 \quad \mbox{for any} \ \phi(x)\in C^\infty_0(\mathbb B_{10}).
\label{show}
\end{align}

Recall that the weak solution $u_\epsilon(x)$ of (\ref{mainhalf}) for the Dirichlet boundary condition is given as
\begin{align}
\int_{\mathbb B_{10}^+} {A}(\frac{x}{\epsilon}) \nabla u_\epsilon\cdot \nabla \phi(x) \, dx=\int_{\partial\mathbb B_{10}^+\cap \{x_d= 0\} } {A}(\frac{x}{\epsilon}) \nabla u_\epsilon\cdot n \phi(x) \, d\sigma
\label{equiv}
\end{align}
for any $\phi(x)\in C^\infty_0(\mathbb B_{10}^+)$. By the assumption of $A(y)$ in (\ref{matrix}), the equation (\ref{equiv}) is equivalent as
\begin{align}
\int_{\mathbb B_{10}^+}\sum^{d-1}_{i,j=1} a_{ij}\frac{\partial u_\epsilon}{\partial x_j}\frac{\partial\phi}{\partial x_i}\,dx+\int_{\mathbb B_{10}^+} a_{dd}\frac{\partial u_\epsilon}{\partial x_d}\frac{\partial \phi}{\partial x_d}\, dx=-
\int_{\partial\mathbb B_{10}^+\cap \{x_d= 0\} } a_{dd} D_d u_\epsilon \phi \, d\sigma.
\label{Diri1}
\end{align}

In order to show (\ref{show}), we split the integral in the left hand side of (\ref{show}) in $\mathbb B_{10}$ into the integrations in
$\mathbb B_{10}^+$ and  $\mathbb B_{10}^-$, where $\mathbb B_{10}^-=\mathbb B_{10}\cap \{x_d< 0\}$. For any $\phi(x)\in C^\infty_0(\mathbb B_{10})$, by the definition of  $\tilde{A}(\frac{x}{\epsilon})$ and $v_\epsilon(x)$,
we have
\begin{align}
\int_{\mathbb B_{10}} \tilde{A}(\frac{x}{\epsilon}) \nabla v_\epsilon\cdot \nabla \phi(x) \, dx
=&\int_{\mathbb B_{10}^+} \tilde{A}(\frac{x}{\epsilon}) \nabla v_\epsilon\cdot \nabla \phi(x) \, dx+\int_{\mathbb B_{10}^-} \tilde{A}(\frac{x}{\epsilon}) \nabla v_\epsilon\cdot \nabla \phi(x) \, dx \nonumber \\
=&\int_{\mathbb B_{10}^+} {A}(\frac{x}{\epsilon}) \nabla u_\epsilon\cdot \nabla \phi(x) \, dx+\int_{\mathbb B_{10}^-} \tilde{A}(\frac{x}{\epsilon}) \nabla v_\epsilon\cdot \nabla \phi(x) \, dx. \label{Diri2}
\end{align}
Applying the integration by parts and change of variables to  transform the integration from $\mathbb B_{10}^-$ to $\mathbb B_{10}^+$, we can show that
\begin{align}
\int_{\mathbb B_{10}^-} \tilde{A}(\frac{x}{\epsilon}) \nabla v_\epsilon\cdot \nabla \phi(x) \, dx
=&- \int_{\mathbb B_{10}^-}\sum^{d-1}_{i,j=1} a_{ij}(\frac{x'}{\epsilon}, \frac{-x_d}{\epsilon})\frac{\partial u_\epsilon(x', -x_d)}{\partial x_j}\frac{\partial\phi}{\partial x_i}\,dx\nonumber \\ &-\int_{\mathbb B_{10}^-} a_{dd}(\frac{x'}{\epsilon}, \frac{-x_d}{\epsilon})\frac{\partial u_\epsilon(x', -x_d)}{\partial x_d}\frac{\partial \phi}{\partial x_d}\, dx \nonumber \\
=&- \int_{\mathbb B_{10}^+}\sum^{d-1}_{i,j=1} a_{ij}(\frac{x'}{\epsilon}, \frac{x_d}{\epsilon})\frac{\partial u_\epsilon}{\partial x_j}\frac{\partial\phi(x', -x_d)}{\partial x_i}\,dx \nonumber \\&-\int_{\mathbb B_{10}^+} a_{dd}(\frac{x'}{\epsilon}, \frac{x_d}{\epsilon})\frac{\partial u_\epsilon}{\partial x_d}\frac{\partial \phi(x', -x_d)}{\partial x_d}\, dx\nonumber \\
=&\int_{\partial\mathbb B_{10}^+\cap \{x_d= 0\} } a_{dd} D_d u_\epsilon \phi \, d\sigma,
\label{Diri3}
\end{align}
where we have used (\ref{Diri1}).
Then (\ref{show}) is verified by combining the estimates (\ref{Diri1}), (\ref{Diri2}) and (\ref{Diri3}). Thus, $v_\epsilon(x)$ is the solution of (\ref{newequ}).

Thanks to Corollary 1, we have the approximate three-ball inequality for $v_\epsilon$.
For $0<R_1<R_2<\frac{\lambda R_3}{4}<\frac{\lambda}{4}$, we obtain that
\begin{align}
\|v_\epsilon\|_{L^\infty(\mathbb B_{R_2})}\leq C\big\{\frac{ R_2}{R_3} \|v_\epsilon\|^\beta_{L^\infty(\mathbb B_{R_1})} \|v_\epsilon\|^{1-\beta}_{L^\infty(\mathbb B_{R_3})} +\frac{R^2_3}{R^2_1} [\epsilon \ln (\epsilon^{-1}+2)]^\beta \|v_\epsilon\|_{L^\infty(\mathbb B_{R_3})}\big\},
\label{sick}
\end{align}
where $\beta=\frac{\ln \frac{\lambda R_3}{2R_2}} {\ln \frac{R_3}{R_1} }$ and $C$ depends on $\lambda$ and $(\tau, \mu)$. Since $v_\epsilon$ is an odd extension of $u_\epsilon$, the estimates (\ref{halfspace}) for the Dirichlet boundary conditions in half balls follow from (\ref{sick}).

For the Neumann boundary conditions, we do an even extension.
We define a new function $v_\epsilon(x)$ as follows,
\begin{align}
v_\epsilon(x)=\left\{ \begin{array}{llr}
u_\epsilon(x', \  x_d) \quad \quad &\mbox{for} \ x_d\geq 0, \nonumber \medskip \\
u_\epsilon(x', \  -x_d)  \quad \quad &\mbox{for} \ x_d< 0.
\end{array}
\right.
\end{align}
The solution $u_\epsilon(x)$ in (\ref{mainhalf}) for the Neumann boundary conditions is Lipschitz continuous, e.g. \cite{KLS1}, \cite{AS}.
Hence, $v_\epsilon(x)$  is Lipschitz continuous and is in $H^1(\mathbb B_{10})$. We also claim that $v_\epsilon(x)$ is the solution for
\begin{align}
-{\rm div}\big(\tilde{A}(\frac{x}{\epsilon})\nabla v_\epsilon\big)=0 \quad \quad \mbox{in} \ \mathbb B_{10}.
\label{newequ2}
\end{align}
Thus, we need to show that
\begin{align}
\int_{\mathbb B_{10}} \tilde{A}(\frac{x}{\epsilon}) \nabla v_\epsilon\cdot \nabla \phi(x) \, dx=0 \quad \mbox{for any} \ \phi(x)\in C^\infty_0(\mathbb B_{10}).
\label{Neum1}
\end{align}

The weak solution $u_\epsilon(x)$ of (\ref{mainhalf}) for the Neumann boundary condition is given as
\begin{align}
\int_{\mathbb B_{10}^+} {A}(\frac{x}{\epsilon}) \nabla u_\epsilon\cdot \nabla \phi(x) \, dx=0
\end{align}
for any $\phi(x)\in C^\infty_0(\mathbb B_{10}^+)$. By the definition of $A(y)$, it also holds that
\begin{align}
\int_{\mathbb B_{10}^+}\sum^{d-1}_{i,j=1} a_{ij}\frac{\partial u_\epsilon}{\partial x_j}\frac{\partial\phi}{\partial x_i}\,dx+\int_{\mathbb B_{10}^+} a_{dd}\frac{\partial u_\epsilon}{\partial x_d}\frac{\partial \phi}{\partial x_d}\, dx=0.
\label{Neum2}
\end{align}
We split the integral in the left hand side of (\ref{Neum1}) in $\mathbb B_{10}$ into the integrations in
$\mathbb B_{10}^+$ and  $\mathbb B_{10}^-$ as
\begin{align}
\int_{\mathbb B_{10}} \tilde{A}(\frac{x}{\epsilon}) \nabla v_\epsilon\cdot \nabla \phi(x) \, dx
=&\int_{\mathbb B_{10}^+} \tilde{A}(\frac{x}{\epsilon}) \nabla v_\epsilon\cdot \nabla \phi(x) \, dx+\int_{\mathbb B_{10}^-} \tilde{A}(\frac{x}{\epsilon}) \nabla v_\epsilon\cdot \nabla \phi(x) \, dx \nonumber \\
=&\int_{\mathbb B_{10}^+} {A}(\frac{x}{\epsilon}) \nabla u_\epsilon\cdot \nabla \phi(x) \, dx+\int_{\mathbb B_{10}^-} \tilde{A}(\frac{x}{\epsilon}) \nabla v_\epsilon\cdot \nabla \phi(x) \, dx.
\label{Neum3}
\end{align}
By (\ref{Neum2}), the definition of $\tilde{A}(\frac{x}{\epsilon})$, $v_\epsilon(x)$ and the change of variables, we can show that
\begin{align}
\int_{\mathbb B_{10}^-} \tilde{A}(\frac{x}{\epsilon}) \nabla v_\epsilon\cdot \nabla \phi(x) \, dx
=&\int_{\mathbb B_{10}^-}\sum^{d-1}_{i,j=1} a_{ij}(\frac{x'}{\epsilon}, \frac{-x_d}{\epsilon})\frac{\partial u_\epsilon(x', -x_d)}{\partial x_j}\frac{\partial\phi}{\partial x_i}\,dx\nonumber \\ &+\int_{\mathbb B_{10}^-} a_{dd}(\frac{x'}{\epsilon}, \frac{-x_d}{\epsilon})\frac{\partial u_\epsilon(x', -x_d)}{\partial x_d}\frac{\partial \phi}{\partial x_d}\, dx \nonumber \\
=& \int_{\mathbb B_{10}^+}\sum^{d-1}_{i,j=1} a_{ij}(\frac{x'}{\epsilon}, \frac{x_d}{\epsilon})\frac{\partial u_\epsilon}{\partial x_j}\frac{\partial\phi(x', -x_d)}{\partial x_i}\,dx \nonumber \\&+\int_{\mathbb B_{10}^+} a_{dd}(\frac{x'}{\epsilon}, \frac{x_d}{\epsilon})\frac{\partial u_\epsilon}{\partial x_d}\frac{\partial \phi(x', -x_d)}{\partial x_d}\, dx\nonumber \\
=&0
\label{Neum4}
\end{align}
for any $\phi(x)\in C^\infty_0(\mathbb B_{10})$. Therefore, the combination of (\ref{Neum2}), (\ref{Neum3}) and (\ref{Neum4}) verifies  the claim  and  $v_\epsilon(x)$ is the solution in (\ref{newequ2}).
Following the same strategy as the case of the Dirichlet boundary conditions, the estimate (\ref{halfspace}) is arrived for equations with the Neumann boundary conditions.
\end{proof}

Following the ideas in Corollary 3 and Theorem 2, we can easily show the proof of Corollary \ref{corhalf}. For the completeness of the presentation, we present the main ideas in the argument.
\begin{proof}[Proof of Corollary \ref{corhalf}]
We consider a new function $v_{\epsilon}(x, t)$ as
\begin{align}
v_{\epsilon}(x, t)= e^{\sqrt{\lambda_{\epsilon, k}} t  }u_\epsilon(x) \quad \quad \mbox{in}  \ \mathbb B^+_{10}\times (-10, \ 10).
\label{newde}
\end{align}
Then the new function $v_{\epsilon}(x, t)$ satisfies the homogenization equation
\begin{align}
\left \{ \begin{array}{rlr}
-{\rm div}\big(A(\frac{x}{\epsilon})\nabla v_\epsilon\big)-\partial^2_t v_\epsilon=0 \quad \quad &\mbox{in} \ \mathbb B_{10}^+\times (-10, \ 10),  \medskip \\
v_\epsilon=0 \ \mbox{or} \ \frac{\partial v_\epsilon}{\partial \nu}=0 \quad \quad &\mbox{on} \ \big\{\partial\mathbb B_{10}^+\cap \{x_d= 0\}\big\}\times (-10, \ 10).
\end{array}
\right.
\label{eigenhalf1}
\end{align}

We denote the new coefficient matrix in $\mathbb B_{10}^+\times (-10, \ 10)$ as
\begin{align*}
\bar{A}(\frac{x}{\epsilon}, t) =\begin{pmatrix}
\big(a_{ij}(\frac{x}{\epsilon})\big)_{d\times d}  & 0 \medskip\\
0 & 1
\end{pmatrix}.
\end{align*}

Thus, the equation in (\ref{eigenhalf1}) can be written as
\begin{align}
-{\rm div}\big(\bar A(\frac{x}{\epsilon}, t)\nabla v_\epsilon\big)=0\quad \quad \mbox{in} \   \mathbb B_{10}^+\times (-10, \ 10).
\end{align}
We do an even extension for the matrix $\bar{A}(\frac{x}{\epsilon}, t)$ across the half space $\{(x, t)|x_d=0\}$ and write the new coefficient matrix  as $\tilde{\bar A}(\frac{x}{\epsilon}, t)$.
This  coefficient matrix $\tilde{\bar A}(\frac{x}{\epsilon}, t)$ still satisfies the conditions (\ref{ellip}), (\ref{perio}) and (\ref{holder}).   As in the proof of Theorem 2, we do an odd extension for $v_\epsilon$ in (\ref{eigenhalf1}) for the Dirichlet boundary conditions across the half space
 $\{(x, t)|x_d=0\}$ and write it as $\tilde{v}_\epsilon(x, t)$.
We identify the half ball with radius $r$ in the cylinder $\mathbb B_{10}^+\times (-10, \ 10)$ as  $\mathbb B_{r}^+\times (-r, \ r)$. Recall that   $\mathbb B_{r}\times (-r, \ r)$ is written as $\tilde{\mathbb {B}}_{r}$.  We can check as in the proof of Theorem 2 that $\tilde{v}_\epsilon(x, t)$ satisfies the equation
\begin{align}
-{\rm div}\big(\tilde{\bar A}(\frac{x}{\epsilon}, t)\nabla \tilde{v}_\epsilon\big)=0 \quad \quad \mbox{in} \ \tilde{\mathbb {B}}_{10}.
\label{another}
\end{align}
Thanks to Corollary 1, the solutions $\tilde{v}_\epsilon$ in (\ref{another}) satisfy the three-ball inequality
\begin{align}
\|\tilde{v}_\epsilon\|_{L^\infty(\tilde{\mathbb {B}}_{R_2})}\leq C\big\{\frac{ R_2}{R_3} \|\tilde{v}_\epsilon\|^\beta_{L^\infty(\tilde{\mathbb {B}}_{R_1})} \|\tilde{v}_\epsilon\|^{1-\beta}_{L^\infty(\tilde{\mathbb {B}}_{R_3})} +\frac{R^2_3}{R^2_1} [\epsilon \ln (\epsilon^{-1}+2)]^\beta \|\tilde{v}_\epsilon\|_{L^\infty(\tilde{\mathbb {B}}_{R_3})}\big\}
\end{align}
for $0<R_1<R_2<\frac{\lambda R_3}{4}<\frac{\lambda}{4}$ with $\beta=\frac{\ln \frac{\lambda R_3}{2R_2}} {\ln \frac{R_3}{R_1} }$. Since we have done an odd extension for $\tilde{v}_\epsilon$,  the following inequality holds for ${v}_\epsilon(x, t)$,
\begin{align}
\|{v}_\epsilon\|_{L^\infty\big({\mathbb {B}}^+_{R_2}\times (-R_2, \ R_2)\big)}\leq &C\big\{\frac{ R_2}{R_3} \|{v}_\epsilon\|^\beta_{L^\infty\big({\mathbb {B}}^+_{R_1}\times (-R_1, \ R_1)\big)} \|{v}_\epsilon\|^{1-\beta}_{L^\infty\big(\tilde{\mathbb {B}}^+_{R_3}\times (-R_3, \ R_3)\big)} \nonumber \\ &+\frac{R^2_3}{R^2_1} [\epsilon \ln (\epsilon^{-1}+2)]^\beta \|{v}_\epsilon\|_{L^\infty\big({\mathbb {B}}^+_{R_3}\times (-R_3, \ R_3)\big)}\big\}.
\end{align}
By the definition of ${v}_\epsilon(x, t)$ in (\ref{newde}), we further obtain that
\begin{align}
\|u_\epsilon\|_{L^\infty(\mathbb B^+_{R_2})}&\leq C e^{2R_3 \sqrt{\lambda_{\epsilon, k}} }\big\{\frac{ R_2}{R_3} \|u_\epsilon\|^\beta_{L^\infty(\mathbb B^+_{R_1})} \|u_\epsilon\|^{1-\beta}_{L^\infty(\mathbb B^+_{R_3})} \nonumber \\ &+\frac{R^2_3}{R^2_1} [\epsilon \ln (\epsilon^{-1}+2)]^\beta \|u_\epsilon\|_{L^\infty(\mathbb B^+_{R_3})}\big\}.
\end{align}
Thus, (\ref{lastcon}) is arrived for solutions of (\ref{eigenhalf}) with the Dirichlet boundary conditions. For the equation (\ref{eigenhalf1}) with the Neumann boundary conditions, we do an even extension for $v_\epsilon(x, t)$. As in the proof of Theorem 2, the approximate three-ball inequality (\ref{lastcon}) can be obtained.

\end{proof}

\end{document}